\documentclass[12pt]{amsart}
\usepackage{amssymb,amsmath,bm}
\usepackage{enumitem}
\usepackage[abbrev,lite,nobysame]{amsrefs}
\usepackage[margin=1in]{geometry}
\usepackage{times}
\usepackage{color}

\def\eps{\varepsilon}
\def\e{{\rm e}}
\def\dist{{\rm dist}}
\def\kur{{\rm kur}}
\def\dim{{\rm dim}}

\def\div{{\rm div}}
\def\d{{\rm d}}
\def\ddt {{\frac{\d}{\d t}}}
\def \l {\langle}
\def \r {\rangle}
\def\de{{\partial}}
\def\Q {{\mathcal Q}}

\def\R {{\mathbb R}}
\def\A {{\mathcal A}}
\def\E {{\mathcal E}}
\def\N {{\mathbb N}}
\def\H {{\mathbb H}}
\def\V {{\mathbb V}}
\def\B {{\mathbb B}}
\def\D {{\mathbb D}}
\def\K {{\mathbb K}}
\def\KK {{\mathcal K}}
\def\dom {{\mathcal D}}

\newtheorem{proposition}{Proposition}[section]
\newtheorem{theorem}[proposition]{Theorem}
\newtheorem{corollary}[proposition]{Corollary}
\newtheorem{lemma}[proposition]{Lemma}
\theoremstyle{definition}

\newtheorem{remark}[proposition]{Remark}

\numberwithin{equation}{section}

\title[Singular limits of Voigt models in fluid dynamics]
{Singular limits of Voigt models in fluid dynamics}

\author[M. Coti Zelati and C.G. Gal ]
{Michele Coti Zelati and Ciprian G. Gal}

\address{Department of Mathematics, University of Maryland
\newline\indent   College Park, MD 20742, USA}
\email{micotize@umd.edu {\rm (M.\ Coti Zelati)}}

\address{Florida International University - Department of Mathematics
\newline\indent Miami, FL 33199, USA}
\email{cgal@fiu.edu {\rm (C.G.\ Gal)}}



\begin{document}

\begin{abstract}
We investigate the long-term behavior, as a certain regularization parameter vanishes, of the
three-dimensional Navier-Stokes-Voigt model of a viscoelastic incompressible fluid. We prove the existence
of global and exponential attractors of optimal regularity.  We then derive explicit upper bounds for
the dimension of these attractors in terms of the three-dimensional Grashof number and the regularization
parameter. Finally, we also prove convergence of the (strong) global attractor of the 3D Navier-Stokes-Voigt
model to the (weak) global attractor of the 3D Navier-Stokes equation. Our analysis improves
and extends recent results obtained by Kalantarov and Titi in \cite{KT}.
\end{abstract}

\maketitle


\section{Introduction}
\noindent Regularized fluid equations in hydrodynamics play a key role in understanding turbulent phenomena
in science. In recent years, many regularized equations have been proposed for the purpose of direct numerical
simulations of turbulent incompressible flows modeled by the Navier-Stokes (NSE) equations \cite{HLT}
(for models of two-phase incompressible fluid flows, we refer the reader to \cite{GMe} for further discussion and results).
One such regularized model is the Navier-Stokes-Voigt (NSV) equations introduced by Oskolkov \cite{Osk1} as a
model for the motion of a linear, viscoelastic, incompressible fluid.

Unlike the 3D Navier-Stokes equations, existence and uniqueness of globally-defined smooth solutions can
be rigorously proven for the 3D NSV equations, as well as the fact that the latter recovers the NSE in the limit
as a certain length scale $\alpha$ goes to zero \cites{KT, HLT}. This property was already noted for other
important regularized models such as the Navier-Stokes $\alpha$-model and the Leray $\alpha$-model \cites{CTV, CTV2}.
Moreover, the robust analytical properties of the 3D NSV model ensure the computability of solutions and
the stability of numerical schemes. Finally, there is evidence that the 3D NSV with a small regularization
parameter enjoys similar statistical properties as the 3D NSE \cites{KLT, RT}. In this connection, understanding
the long-time behavior of solutions to the 3D NSV equation is crucial. The global existence and uniqueness of
strong solutions for the 3D NSV model is by now a classical result \cite{Osk1}. The existence of the global attractor
of optimal regularity as well as an upper bound on its dimension were established in \cite{KT}. In the same paper,
estimates for the number of asymptotic determining modes of solutions of the 3D NSV problem were also derived.

Our objective in the present contribution is to establish the following results: (a) the existence of exponential
attractors $\E_\alpha$ for the 3D NSV model and (b) provide explicit upper bounds on their fractal dimension;
(c) motivated by comparable results established in \cites{CTV, CTV2} for other regularized models, prove
convergence results for the corresponding global attractors $\A_\alpha$ as the regularization parameter $\alpha$
goes to zero. We recall that the exponential attractor always contains the global attractor and also attracts bounded
subsets of the energy phase-space at an exponential rate, which makes it a more useful object in numerical simulations
than the global attractor.

While the first issue (a) has already been treated in \cite{GMe}*{Section 6} to some extent
as a special case of a general family of regularized models, in this contribution we also wanted to rephrase the results
associated with the 3D NSV into $\alpha$-dependent spaces, in order to simplify the handling of the dependence on the
parameter $\alpha>0$. One feature of this analysis is that we are able to obtain the optimal regularity of the global attractor
in \emph{one} step as opposed to employing a complicated bootstrapping procedure as in \cite{KT}. We emphasize that
the former procedure has proved quite useful in the treatment of other regularized models for the 3D NSE (cf. \cite{GMe}*{Section 6}).

Concerning the second issue (b), the best (up to date) estimate for the fractal dimension of the global attractor was found in \cite{KT}.
There it was shown that it is asymptotically growing like $c\alpha^{-6}$, for some constant $c>0$ and thus it blows up
as $\alpha$ goes to zero. Our present improvement consists in deriving a better upper bound for the dimension of both global and exponential attractors associated with the 3D NSV model. Based on the
Constantin-Foias trace formula \cite{LADY}, we find an upper bound for the fractal (box-counting) dimension of the global
attractor growing like $c\alpha^{-3}$, where $c$ is an explicitly computable constant \emph{independent} of $\alpha$. 
We refer the reader to Section 5 and to Remark \ref{important} for further details.

Following a scheme introduced by Eden, Foias and Kalantarov \cite{Eden2}, we finally show that the 3D NSV model has
an exponential attractor $\E_\alpha$ which admits the same fractal dimension estimate as for the corresponding global
attractor $\A_\alpha$. These results further justify the use of the Navier-Stokes-Voigt equations as an inviscid regularization
of the 3D NSE, in particular for numerical computations and simulations. For the final issue (c), based on the concept of multivalued
semiflows \cites{CZ,MV}, we derive results on the convergence of the (strong) global attractors for the Voigt model to the (weak) global
attractor for the 3D NSE, as $\alpha$ goes to zero. We also derive conditions for the weak global attractor of the NSE to be
strong, in terms of the topologies of the above mentioned convergences.

The paper is divided into five main sections: Section 2 provides the abstract setting for the 3D NSV model.
Then we establish the existence of global and exponential attractors in Sections 3 and 4, respectively. In
Section 5, we provide explicit bounds on the fractal dimension of these attractors and in Section 6 we discuss
the convergence of global attractors as $\alpha$ goes to zero.

\section{Abstract setting}

\noindent We start by introducing the models under study and fix the abstract functional setting
typical of the Navier-Stokes equations \cites{CONFO,S,T1,T2}. The notation will include a scale of $\alpha$-dependent
spaces, which arise naturally as finite energy spaces for the Navier-Stokes-Voigt equations.

\subsection{The fluid equations}
Given a bounded domain $\Omega\subset \R^3$ with boundary
$\de\Omega$ of class $C^{2}$, we consider the Navier-Stokes equations ruling the velocity vector
$u=u(\boldsymbol{x},t)$ and the pressure $p=p(\boldsymbol{x},t)$ of a homogeneous
incompressible fluid. In dimensionless form, the equations read
\begin{equation}\label{eq:NS}
\begin{cases}
u_t-\nu\Delta u + (u\cdot \nabla)u +\nabla p=f,\qquad \boldsymbol{x}\in \Omega, \ t>0,\\
\div\,  u =0, \qquad \boldsymbol{x}\in \Omega, \ t>0,
\end{cases}
\end{equation}
Here $f=f(\boldsymbol{x})$ is an autonomous density force per unit volume
and $\nu>0$ is the constant kinematic viscosity parameter of the fluid.
The system is supplemented with the nonslip  boundary condition
\begin{equation}\label{eq:NSBC}
u(\boldsymbol x,t)|_{\boldsymbol x\in\de \Omega}=0,\qquad t>0,
\end{equation}
and the initial condition
\begin{equation}\label{eq:NSIC}
u(\boldsymbol x,0)=u_0(\boldsymbol x), \qquad \boldsymbol{x}\in \Omega,
\end{equation}
for an assigned divergence-free function $u_0$.
In the same setting, we also take into account a modified version
of equations \eqref{eq:NS}-\eqref{eq:NSIC} given by
\begin{equation}\label{eq:NSV}
\begin{cases}
u_t-\alpha^2 \Delta u_t-\nu\Delta u + (u\cdot \nabla)u +\nabla p=f,\qquad \boldsymbol{x}\in \Omega, \ t>0,\\
\div\,  u =0, \qquad \boldsymbol{x}\in \Omega, \ t>0,
\end{cases}
\end{equation}
known in the literature as Navier-Stokes-Voigt\footnote{In the literature Voigt is sometimes spelled Voight.} equations.
The parameter $\alpha\in(0,1]$ above is a length scale parameter characterizing
the elasticity of the fluid. From the physical viewpoint, its presence
is a consequence of a modification of the Cauchy stress tensor, accounting for viscoelastic effects
of Kelvin-Voigt type. It is worth noticing that when $\alpha=0$, it is possible to recover, at least formally,
the classical Navier-Stokes equations \eqref{eq:NS}.

\subsection{Mathematical setting}\label{sub:mathset}
For $p\in[1,\infty]$ and
$k\in\N$, we denote
by $\mathbf{L}^p(\Omega)=\{L^p(\Omega)\}^3$, $\mathbf{H}^k(\Omega)=\{H^k(\Omega)\}^3$,
and $\mathbf{H}_0^k(\Omega)=\{H_0^k(\Omega)\}^3$
the usual Lebesgue and Sobolev spaces of vector-valued functions on $\Omega$. Setting
$$
\mathcal{V}=\big\{u\in C^\infty_0(\Omega, \R^3): \, \div\, u =0 \big\},
$$
we consider the usual Hilbert space associated with the Navier-Stokes equations
\begin{align*}
\H=\text{closure of } \mathcal{V} \text{ in } \mathbf{L}^2(\Omega),
\end{align*}
where $|\cdot|$ and  $\l\cdot,\cdot\r$ indicate its norm and scalar product, respectively.
Calling
$$
P:\mathbf{L}^2(\Omega)=\H\oplus \H^\perp\to \H
$$
the Leray orthogonal projection, the Stokes operator is defined as
$$
A=-P\Delta, \qquad \dom (A)=\mathbf{H}^2(\Omega)\cap  \V,
$$
with
$$
\V=\text{closure of } \mathcal{V} \text{ in } \mathbf{H}^1(\Omega).
$$
It is well known that the operator $A$ is self-adjoint and strictly positive. Moreover, $\dom (A^{1/2})=\V$ and
$$
\|u\|=|\nabla u|=|A^{1/2}u|, \qquad \forall u\in \V.
$$
Also, we indicate by $\V^*$ the
dual space of $\V$, endowed with the usual dual norm $\|\cdot\|_*$, and for which
the duality with $\V$ will also be written as $\l\cdot,\cdot\r$.
For $\alpha\in [0,1]$ and $s\in \R$, we define
the scale of Hilbert spaces
$$
\V_\alpha^s= \dom (A^{s/2}),
$$
endowed with scalar product
$$
\l u,v\r_{\V_\alpha^s}= \l A^{(s-1)/2}u, A^{(s-1)/2}v\r+\alpha^2\l A^{s/2}u,A^{s/2}v\r,
$$
and norm
$$
\|u\|^2_{\V_\alpha^s}= |A^{(s-1)/2}u|^2+\alpha^2|A^{s/2}u|^2.
$$
It is understood that, when $s<1$, the space $\V_\alpha^s$ consists of the completion
of $\H$ with respect to the above mentioned norm.
When $s=1$, we will simply write $\V_\alpha$ in place of $\V_\alpha^1$ and
$$
\|u\|^2_{\V_\alpha}= |u|^2+\alpha^2\|u\|^2,
$$
in agreement with the standard notation used for the Navier-Stokes equations.
Moreover, we have the so-called Poincar\'e inequalities
\begin{equation}\label{eq:poinc}
\begin{aligned}
&|u|\leq \frac{1}{\sqrt{\lambda_1}}\|u\|, \qquad \forall u\in \V,\\
&\|u\|\leq \frac{1}{\sqrt{\lambda_1}}|Au|, \qquad \forall u\in \dom(A),
\end{aligned}
\end{equation}
where $\lambda_1>0$ is the first eigenvalue of the Stokes operator $A$.
We set
$$
B(u,v)=P\big[(u\cdot \nabla)v\big].
$$
From the bilinear operator $B(\cdot, \cdot)$, we can define the trilinear form
$$
b(u,v,w)=\l B(u,v),w\r,
$$
which is continuous on $\V\times \V\times \V$ and satisfies
\begin{alignat}{2}
&b(u,v,v)=0, \qquad & &\forall u,v\in \V, \label{eq:b1}\\
&|b(u,v,w)|\leq c\|u\|  \|v\| \|w\|^{1/2}|w|^{1/2}, \qquad & &\forall u,v,w\in \V, \label{eq:b2}\\
&|b(u,v,w)|\leq c\|u\| \|v\|^{1/2}|Av|^{1/2} |w|, \qquad & &\forall u, w\in \V, v\in \dom(A).\label{eq:b3}
\end{alignat}
Above and in the rest of the paper, $c$ will denote a dimensionless scale invariant positive constant which might depend on
the shape of the domain $\Omega$ and other parameters of the problem, such as the viscosity $\nu$ or the forcing term $f$. Unless
otherwise stated, it will be \emph{independent} of $\alpha$.

\begin{remark}
We emphasize that all the results proven in this article are also valid in a more general setting,
when $\Omega$ is a compact Riemannian manifold with or without boundary and in the presence of other boundary conditions.
We refer the reader to \cites{HLT, GMe} for more details.
\end{remark}

System \eqref{eq:NS}--\eqref{eq:NSIC} can be rewritten as an abstract evolution equation of the form
\begin{equation}\label{eq:NSABS}
\begin{cases}
\dot{u} +\nu Au+B(u,u)=g,\\
u(0)=u_0,
\end{cases}
\end{equation}
where $g=Pf$. In the exact same way, \eqref{eq:NSV} turns into
\begin{equation}\label{eq:NSVABS}
\begin{cases}
\dot{u}+\alpha^2A\dot{u} +\nu Au+B(u,u)=g,\\
u(0)=u_0.
\end{cases}
\end{equation}

\subsection{The dynamical system generated by the Voigt model}
In the first part of this article, we will be concerned with the study of the Voigt regularized
model \eqref{eq:NSVABS}, for an arbitrary but fixed value of $\alpha\in(0,1]$. It has been known
since \cite{Osk1} that, for every $\alpha\in (0,1]$, the system \eqref{eq:NSVABS}
generates a strongly continuous semigroup
$$
S_\alpha(t):\V_\alpha\to \V_\alpha.
$$
Moreover, $S_\alpha(t)$ satisfies the following continuous dependence estimate.
\begin{lemma}
Let $u_0,v_0\in \V_\alpha$. Then
\begin{equation}\label{eq:contdep}
\|S_\alpha(t)u_0-S_\alpha(t)v_0\|_{\V_\alpha}\leq \e^{c\frac{(r+t)}{\alpha^2}} \|u_0-v_0\|_{\V_\alpha}, \qquad t\geq 0,
\end{equation}
whenever
$$
\|u_0\|_{\V_\alpha},\|v_0\|_{\V_\alpha}\leq r.
$$
\end{lemma}

\begin{proof}
Let $u(t)=S_\alpha(t)u_0$ and $v(t)=S_\alpha(t)v_0$. The difference $w=u-v$ satisfies
$$
\begin{cases}
\dot{w}+\alpha^2A\dot{w} +\nu Aw+B(u,w)+B(w,v)=0,\\
w(0)=u_0-v_0.
\end{cases}
$$
Testing the above equation by $w$ in $\H$ gives
$$
\frac12\ddt \big[|w|^2+\alpha^2\|w\|^2\big]+\nu \|w\|^2=-\l B(w,v),w\r.
$$
The right hand side of the above equation can be estimated through \eqref{eq:b2}, to obtain
$$
|\l B(w,v),w\r|  \leq c  \|v\|\|w\|^2\leq \nu\|w\|^2 + c  \|v\|^2\|w\|^2.
$$
Therefore,
$$
\ddt \|w\|_{\V_\alpha}^2\leq c  \|v\|^2\|w\|^2 \leq \frac{c}{\alpha^2}\|v\|^2\|w\|_{\V_\alpha}^2.
$$
Thanks to the integral estimate \eqref{eq:absorb2} below, we have
$$
\int_0^t \|v(s)\|^2\d s\leq c(\|v_0\|^2_{\V_\alpha}+t),\qquad \forall t\geq 0,
$$
and the conclusion follows from the standard Gronwall lemma.
\end{proof}

\begin{remark}
It is expected that the estimate \eqref{eq:contdep} blows up as $\alpha\to 0$.
The limit case would correspond to prove a continuous dependence result for
weak solutions to the three-dimensional Navier-Stokes equations, a  longstanding
open problem that seems to be out of reach at the moment.
\end{remark}

\section{Global attractors}
\noindent We now deal with the dissipative features of the semigroup $S_\alpha(t)$, for $\alpha\in (0,1]$ arbitrary but fixed.
Specifically, we prove the existence of the global attractor and analyze its regularity in higher Sobolev spaces.
Define (twice) the energy of the system as
$$
E_\alpha(t)=\|S_\alpha(t)u_0\|^2_{\V_\alpha}=|u(t)|^2+\alpha^2\| u(t)\|^2.
$$
Thanks to the Poincar\'e inequality \eqref{eq:poinc}, it is straightforward to check that
\begin{equation}\label{eq:ineq}
\alpha^2\|u\|^2\leq E_\alpha \leq \frac{1+\lambda_1\alpha^2}{\lambda_1} \|u\|^2.
\end{equation}

\subsection{A first dissipative estimate}
We begin to show that
the trajectories originating from any given bounded set eventually fall, uniformly in time,
into a bounded absorbing set $\B_\alpha\subset \V_\alpha$. The proof of its existence
is based on the following standard argument, which can be made rigorous by means of a
Galerkin approximation procedure.
We multiply \eqref{eq:NSVABS} by $u$ in $\H$, to obtain
\begin{equation}\label{eq:zerodiff}
\frac12 \ddt E_\alpha +\nu\|u\|^2=\l g,u\r\leq \frac{\nu}{2}\|u\|^2+\frac{|g|^2}{2\nu\lambda_1}.
\end{equation}
Hence,
\begin{equation}\label{eq:firstdiff}
\ddt E_\alpha +\nu\|u\|^2\leq \frac{|g|^2}{\nu\lambda_1}.
\end{equation}
Thanks to \eqref{eq:ineq}, we obtain
$$
\ddt E_\alpha +\kappa_\nu E_\alpha \leq \frac{|g|^2}{\nu\lambda_1},
$$
where
\begin{equation}\label{eq:kappanu}
\kappa_\nu:= \frac{\nu\lambda_1}{1+\lambda_1\alpha^2}.
\end{equation}
The Gronwall lemma then entails
\begin{equation}\label{eq:absorb}
E_\alpha(t)\leq E_\alpha(0)\e^{-\kappa_\nu t}+ \frac{|g|^2}{\nu\kappa_\nu\lambda_1}
=E_\alpha(0)\e^{-\kappa_\nu t}+ \frac{(1+\lambda_1\alpha^2)|g|^2}{\nu^2\lambda_1^2},
\end{equation}
and, using one more time \eqref{eq:firstdiff}, we also obtain the integral estimate
\begin{equation}\label{eq:absorb2}
\nu\int_0^t\|u(s)\|^2\d s\leq \| u_0\|_{\V_\alpha}^2+ \frac{|g|^2}{\nu\lambda_1} t, \qquad \forall t\geq 0.
\end{equation}
Denote
\begin{equation}\label{eq:absball}
\B_\alpha=\left\{w\in \V_\alpha: |w|^2+\alpha^2 \|w\|^2\leq \frac{2(1+\lambda_1\alpha^2)|g|^2}{\nu^2\lambda_1^2}\right\}.
\end{equation}
\begin{proposition}\label{prop:absorb}
The set $\B_\alpha$ is a bounded absorbing set for $S_\alpha(t)$.
\end{proposition}

In more precise terms, given a bounded set $B\subset \V_\alpha$, we have proved that there exists an entering time
$t_B>0$ such that
$$
S_\alpha(t)B\subset \B_\alpha, \qquad \forall t\geq t_B.
$$
In particular, this gives a rough estimate on the asymptotic behavior of the Voigt system, and allows us
to restrict our attention to the absorbing set $\B_\alpha$ itself. Indeed, the long term dynamics of
trajectories departing from $\B_\alpha$ captures necessarily the dynamics of any trajectory, due to the
fact that any trajectory will be absorbed by $\B_\alpha$ in finite time.

\subsection{The semigroup decomposition}\label{sub:dec}
From the dissipative estimate \eqref{eq:absorb} above, it is clear that
\begin{equation}\label{eq:estind}
\sup_{t\geq 0}\sup_{u_0\in \B_\alpha} \| S_\alpha(t)u_0 \|_{\V_\alpha}\leq M_1:=
\left[\frac{3(1+\lambda_1\alpha^2)|g|^2}{\nu^2\lambda_1^2}\right]^{1/2}.
\end{equation}
For $u_0\in \B_\alpha$, we split the solution as
$$
S_\alpha(t)u_0=L_\alpha(t)u_0+K_\alpha(t)u_0
$$
where
$$
v(t)=L_\alpha(t)u_0 \qquad \text{and}\qquad w(t)=K_\alpha(t)u_0
$$
respectively solve
\begin{equation}\label{eq:decay}
\begin{cases}
\dot{v}+\alpha^2A\dot{v} +\nu Av+B(u,v)=0,\\
v(0)=u_0,
\end{cases}
\end{equation}
and
\begin{equation}\label{eq:cpt}
\begin{cases}
\dot{w}+\alpha^2A\dot{w} +\nu Aw+B(u,w)=g,\\
w(0)=0.
\end{cases}
\end{equation}
The proof of the exponential decay of the solution operator $L_\alpha(t)$ follows
word for word the derivation of the dissipative estimate \eqref{eq:absorb} with
$g=0$. We therefore state the result in the following lemma, without proof.
\begin{lemma}\label{lem:decay}
For every $t\geq 0$, there holds
$$
\| L_\alpha(t)u_0\|_{\V_\alpha}^2\leq \|u_0\|^2_{\V_\alpha}\e^{-\kappa_\nu t},
$$
where $\kappa_\nu$ is given by \eqref{eq:kappanu}.
\end{lemma}
Next we show that, for every fixed time, the component related to $K_\alpha(t)u_0$ belongs to a compact
subset of $\V_\alpha$, uniformly as the initial data $u_0$ belongs to the absorbing set $\B_\alpha$, given by
Proposition \ref{prop:absorb}.
\begin{lemma}\label{lem:cpt}
For every $\alpha\in (0,1]$ and $u_0\in\B_\alpha$, we have the estimate
$$
\sup_{t\geq 0}\| K_\alpha(t)u_0\|_{\V_\alpha^2}^2\leq r_\alpha,
$$
where
\begin{equation}\label{eq:ralpha}
r_\alpha:=
\frac{c}{\kappa_\nu} \left[\frac{M_1^6}{\alpha^6\nu^3}+\frac{2}{\nu}|g|^2\right].
\end{equation}
\end{lemma}

\begin{proof}
Multiplying \eqref{eq:cpt} by $Aw$, we are led to the identity
$$
\frac12\ddt \Psi_\alpha +\nu |Aw|^2=\l g,Aw\r-\l B(u,w),Aw\r,
$$
where
$$
\Psi_\alpha(t)=\| K_\alpha(t)u_0\|_{\V_\alpha^2}^2=\| w(t)\|^2+\alpha^2|Aw(t)|^2.
$$
Clearly,
$$
|\l g,Aw\r|\leq \frac{\nu}{4}|Aw|^2 +\frac{1}{\nu}|g|^2.
$$
Also, in view of \eqref{eq:b2} and the Young inequality,
\begin{align*}
|\l B(u,w),Aw\r| \leq c\|u\|\|w\|^{1/2}|Aw|^{3/2}\leq \frac{\nu}{4}|Aw|^2+ c\frac{M_1^6}{\alpha^6\nu^3}.
\end{align*}
Therefore,
$$
\ddt \Psi_\alpha +\nu |Aw|^2\leq c\frac{M_1^6}{\alpha^6\nu^3}+\frac{2}{\nu}|g|^2.
$$
Again, it is easy to see that \eqref{eq:poinc} implies
$$
\nu |Aw|^2\geq \kappa_\nu \Psi_\alpha,
$$
so that
$$
\ddt \Psi_\alpha +\kappa_\nu \Psi_\alpha\leq c\frac{M_1^6}{\alpha^6\nu^3}+\frac{2}{\nu}|g|^2.
$$
The conclusion follows from the Gronwall lemma, noticing that $\Psi_\alpha(0)=0$.
\end{proof}

\subsection{The global attractor and its regularity}
The aim of this section is to prove the existence of a universal attractor for $S_\alpha(t)$ on $\V_\alpha$.
Recall that the universal attractor is the (unique) compact set $\A_\alpha\subset \V_\alpha$, which is at the
same time attracting, in the sense of the Hausdorff
semidistance\footnote{Given a metric space $(X,\d_X)$,
the Hausdorff semidistance $\dist_X(B,C)$ between two  sets $B,C\subset X$ is given by
$$
\dist_X(B,C)=\sup_{b\in B}\inf_{c\in C}\d_X(b,c).
$$}, and fully invariant for $S_\alpha(t)$,
that is, $S_\alpha(t)\A_\alpha=\A_\alpha$ for all $t\geq 0$ (see, e.g., \cites{H,T3}).
Define
$$
\D_\alpha(r)=\big\{u\in \V_\alpha^2: \|u\|^2+\alpha^2|Au|^2\leq r\big\}\cap \B_\alpha.
$$
From Lemmas \ref{lem:decay} and \ref{lem:cpt}, we obtain
\begin{theorem}\label{thm:expattr1}
The set $\D_\alpha(r_\alpha)$ is a compact exponentially attracting set for $S_\alpha(t)$, namely
$$
\dist_{\V_\alpha}(S_\alpha(t)\B_\alpha, \D_\alpha(r_\alpha))\leq \|\B_\alpha\|_{\V_\alpha}\e^{-\frac{\kappa_\nu}{2} t}.
$$
\end{theorem}
It is well known that the existence of an attracting set is equivalent to the existence of the global attractor.
Moreover, the fact that the global attractor is the minimal closed attracting set implies that the following holds.

\begin{corollary}
Let $\alpha\in(0,1]$. There exists the global attractor $\A_\alpha\subset \V_\alpha$ for $S_\alpha(t)$. Moreover,
$\A_\alpha\subset \D_\alpha(r_\alpha)$, and it is therefore bounded in $\V_\alpha^2$.
\end{corollary}

\begin{remark}
The existence of a finite-dimensional global attractor was shown for the first time in \cite{K1}.
The estimates of the fractal and Hausdorff dimensions of the global
attractor were later improved in \cite{KT}, where also an upper bound on the number
of asymptotic determining modes  of the solutions was provided. Concerning the regularity
of the attractor, it can be proved to be as smooth as the forcing term $f$ permits, and
even real analytic, whenever $f$ is analytic as well \cite{KLT}.
\end{remark}

\begin{remark}
Our proof of the existence of the global attractor slightly differs from the one in \cite{KT}. While here
we obtain the optimal regularity in one step (also as in \cite{GMe}*{Section 6}), the method of
\cite{KT} requires a bootstrapping procedure which implies first the regularity in $\V_\alpha^{3/2}$,
then in $\V_\alpha^{5/3}$ and finally in $\V_\alpha^{2}$. Here, we also wanted to rephrase those
results into $\alpha$-dependent spaces, in order to simplify
the handling of the dependence on the parameter $\alpha$.
\end{remark}

\section{Exponential attractors}
\noindent Despite the existence of an exponentially attracting set, quantitative information on the
attraction rate of the global attractor is usually very hard to find, if not out of reach. To overcome
this difficulty, it was introduced in \cite{EFNT} the concept of exponential attractor.
An exponential attractor is a compact positively invariant subset of the phase space of finite
fractal dimension which attracts all trajectories at an exponential rate. Recall that the fractal
dimension of a compact set $K$ in a metric space $X$ is defined by
$$
\dim_XK=\limsup_{\eps \to 0} \frac{\log N(\eps,K)}{\log(1/\eps)},
$$
where $N(\eps,K)$ is the smallest numbers of balls of radius $\eps$ necessary
to cover $K$.
The main result of this section is the following.

\begin{theorem}\label{thm:expattr}
Let $\alpha\in (0,1]$. The dynamical system $S_\alpha(t)$ on $\V_\alpha$ admits an exponential attractor $\E_\alpha$
contained and bounded in $\V^2_\alpha$. Precisely,
\begin{itemize}
	\item $\E_\alpha$ is positively invariant for $S_\alpha(t)$, that is, $S_\alpha(t)\E_\alpha\subset \E_\alpha$ for every $t\geq 0$;
	\item $\dim_{\V_\alpha}\, \E_\alpha<\infty$, that is, $\E_\alpha$ has finite fractal dimension in $\V_\alpha$;
	\item there exist an increasing function $\Q_\alpha:[0,\infty)\to[0,\infty)$ and $\kappa_\alpha>0$ such that,
	for any bounded set $B\subset \V_\alpha$ there holds
	$$
	\dist_{\V_\alpha}(S_\alpha(t)B,\E_\alpha)\leq \Q_\alpha (\|B\|_{\V_\alpha})\e^{-\kappa_\alpha t}.
	$$
\end{itemize}
\end{theorem}
As a byproduct, we have the following.

\begin{corollary}\label{cor:finitefrac}
The global attractor $\A_\alpha$ of $S_\alpha(t)$ has finite fractal dimension in $\V_\alpha$.
\end{corollary}

To the best of our knowledge, the
best estimate for the fractal dimension in $\H$ was found in \cite{KT}, and it is asymptotically growing
like $\alpha^{-6}$. The improvement here is that the fractal dimension in $\V_\alpha$ (and, in turn, with a scaling argument, in $\V$) is
proven to be finite. We discuss a quantitative estimate in the subsequent Section \ref{sec:fractal}.

\subsection{Higher order estimates}
We begin by showing that solutions  to the Navier-Stokes-Voigt system originating from regular initial data
remain regular uniformly in time.

\begin{lemma}\label{lem:regest}
Let $\alpha\in (0,1]$, and assume $u_0\in \D_\alpha(\rho)$ for some $\rho>0$. Then
\begin{equation}\label{eq:regest1}
\|S_\alpha(t)u_0\|^2_{\V^2_\alpha}\leq \rho \e^{-\kappa_\nu t} +r_\alpha, \qquad \forall t\geq 0.
\end{equation}
In particular, we have the uniform estimate
\begin{equation}\label{eq:regest2}
\sup_{t\geq 0}\sup_{u_0\in \D_\alpha(\rho)}\|S_\alpha(t)u_0\|^2_{\V^2_\alpha}\leq \rho+r_\alpha.
\end{equation}
\end{lemma}
\begin{proof}
Taking the scalar product of \eqref{eq:NSVABS} with $Au$ in $\H$, we see that the functional
$$
\Psi_\alpha(t)=\| S_\alpha(t)u_0\|_{\V_\alpha^2}^2=\| u(t)\|^2+\alpha^2|Au(t)|^2
$$
satisfies the identity
$$
\frac12\ddt \Psi_\alpha +\nu |Au|^2=\l g,Au\r-\l B(u,u),Au\r.
$$
The standard estimate
$$
|\l g,Au\r|\leq \frac{\nu}{4}|Au|^2 +\frac{1}{\nu}|g|^2,
$$
together with
\begin{align*}
|\l B(u,u),Au\r| \leq c\|u\|\|u\|^{1/2}|Au|^{3/2}\leq \frac{\nu}{4}|Au|^2+ c\frac{M_1^6}{\alpha^6\nu^3},
\end{align*}
entails
\begin{equation}\label{eq:higherestA}
\ddt \Psi_\alpha +\nu |Au|^2\leq \kappa_\nu r_\alpha.
\end{equation}
Again, it is easy to see that by \eqref{eq:poinc} we have
$$
\nu |Au|^2\geq \kappa_\nu \Psi_\alpha,
$$
so that
$$
\ddt \Psi_\alpha +\kappa_\nu \Psi_\alpha\leq \kappa_\nu r_\alpha.
$$
The conclusion follows from the Gronwall lemma.
\end{proof}
Another way to state the above result (see \cite{KT}) is that the dynamical
system $S_\alpha(t)$ restricted to $\V^2_\alpha$ possesses a bounded
absorbing set.

\subsection{Invariant exponentially attracting sets}
One essential difficulty, when constructing exponential attractors for hyperbolic
equations, is to prove that the exponential attractors attract the bounded subsets of
the whole phase space, and not those starting from a subspace of the phase space only
(typically, consisting of more regular functions). This difficulty was overcome in
\cite{FGMZ} by proving the following transitivity property of the exponential attraction.

\begin{lemma}\label{lem:zelik}
Let $\KK_0,\KK_1,\KK_2\subset \V_\alpha$ be such that
$$
\dist_{\V_\alpha}(S_\alpha(t)\KK_0,\KK_1)\leq C_0\e^{-\omega_0 t}\qquad \text{and}\qquad
\dist_{\V_\alpha}(S_\alpha(t)\KK_1,\KK_2)\leq C_1\e^{-\omega_1 t}
$$
for some $C_0,C_1\geq 0$ and $\omega_0,\omega_1>0$. Assume also that for all
$u_0,v_0\in \bigcup_{t\geq 0}S_\alpha(t)\KK_i$
$$
\|S_\alpha(t)u_0-S_\alpha(t)v_0\|_{\V_\alpha}\leq K\e^{\kappa t}\|u_0-v_0\|_{\V_\alpha},
$$
for some $K\geq 0$ and $\kappa>0$. Then
$$
\dist_{\V_\alpha}(S_\alpha(t)\KK_0,\KK_2)\leq C\e^{-\omega t},
$$
where $C=KC_0+C_1$ and $\omega=\frac{\omega_0\omega_1}{\kappa+\omega_0+\omega_1}$.
\end{lemma}

We saw previously that the set $\D_\alpha(r_\alpha)$ is a regular exponentially attracting set. We wish
to find an invariant set with the same properties. Let $t_e>0$ be the entering time of $\D_\alpha(r_\alpha)$
in the absorbing set $\B_\alpha$, and define
$$
\K_\alpha= \overline{\bigcup_{t\geq t_e} S_\alpha(t) \D_\alpha(r_\alpha)}^{\V_\alpha}.
$$
We claim that $\K_\alpha$ has the required properties. The invariance is fairly straightforward.
Due to the continuity of $S_\alpha(t)$, we have
\begin{align*}
S_\alpha(t)\K_\alpha&=S_\alpha(t)\overline{\bigcup_{\tau\geq t_e} S_\alpha(\tau) \D_\alpha(r_\alpha)}^{\V_\alpha}
\subset \overline{\bigcup_{\tau\geq t_e} S_\alpha(t+\tau) \D_\alpha(r_\alpha)}^{\V_\alpha}\\
&\subset \overline{\bigcup_{\tau\geq t_e} S_\alpha(\tau) \D_\alpha(r_\alpha)}^{\V_\alpha}=\K_\alpha.
\end{align*}
We turn our attention to the regularity of $\K_\alpha$.
\begin{lemma}\label{lem:pain1}
$\K_\alpha$ is a compact set in $\V_\alpha$, bounded in $\V^2_\alpha$. Precisely,
$$
\K_\alpha\subset \D_\alpha(2r_\alpha).
$$
\end{lemma}

\begin{proof}
It suffices to show boundedness in $\V^2_\alpha$ thanks to the compact embedding of $\V^2_\alpha$ into
$\V_\alpha$ and
the fact that $\K_\alpha$ is clearly closed in $\V_\alpha$.
Let $w\in \K_\alpha$. Then, there exists a sequence $t_n\geq t_e$ and
$$
w_{n}\in S_\alpha(t_n) \D_\alpha(r_\alpha)
$$
such that $w_n\to w$ strongly in $\V_\alpha$. Lemma \ref{lem:regest} implies that
$$
\|w_n\|^2_{\V^2_\alpha} \leq 2r_\alpha.
$$
As a consequence, by weak compactness and the uniqueness of the limit we learn that
$$
w_n\rightharpoonup w \qquad \text{weakly in } \V_\alpha^2.
$$
The lower semicontinuity  of the norm with respect to weak convergence finally allows us to conclude that
$$
\|w\|^2_{\V^2_\alpha}\leq \liminf_{n\to\infty}\|w_{n}\|^2_{\V^2_\alpha}\leq 2r_\alpha.
$$
\end{proof}
It is left to show that $\K_\alpha$ is exponentially attracting.
This is possible by exploiting the transitivity property of the exponential attraction, stated above in Lemma
\ref{lem:zelik}.

\begin{lemma}\label{lem:pain3}
There holds
$$
\dist_{\V_\alpha}(S(t)\B_\alpha, \K_\alpha)\leq L_\alpha\e^{-\omega_\alpha t}, \qquad \forall t\geq 0,
$$
for some $L_\alpha\geq 0$ and some $\omega_\alpha>0$. Precisely,
$$
L_\alpha\sim \sqrt{r_\alpha}\sim \frac{1}{\alpha^3}
$$
and
$$
\omega_\alpha\sim \alpha^{2} .
$$
\end{lemma}

\begin{proof}
Clearly,
$$
S_\alpha(t)\D_\alpha(r_\alpha)\subset \K_\alpha, \qquad \forall t\geq t_e.
$$
Therefore, $\K_\alpha$ exponentially attracts $\D_\alpha(r_\alpha)$ at any rate.
Since  $\K_\alpha$ is contained in $\D_\alpha (2r_\alpha)$,
we can write
$$
\dist_{\V_\alpha}(S_\alpha(t)\D_\alpha(r_\alpha),\K_\alpha)\leq c \sqrt{r_\alpha} \e^{-\frac{\kappa_\nu}{2} t}, \qquad \forall t\geq 0,
$$
where $c>0$ is an absolute constant independent of $\alpha$. Moreover, from the continuous dependence
estimate \eqref{eq:contdep}, we find that
$$
\|S_\alpha(t)u_0-S_\alpha(t)v_0\|_{\V_\alpha}\leq c\,\e^{\tilde{c}_\alpha t}\|u_0-v_0\|_{\V_\alpha},
$$
where
$$
u_0,v_0\in \bigcup_{t\geq 0} \big[S_\alpha(t)\B_\alpha
\cup S_\alpha(t)\D_\alpha(r_\alpha)
\cup S_\alpha(t)\K_\alpha \big]\subset 2\B_\alpha
$$
and
$$
\tilde{c}_\alpha\sim \frac{1}{\alpha^{2}}.
$$
Therefore, by the transitivity of the exponential attraction in Lemma \ref{lem:zelik}, we can conclude the proof
of the lemma by using Theorem \ref{thm:expattr1}.
\end{proof}

\subsection{Exponential attractors}
We are now in the position to prove Theorem \ref{thm:expattr}, following the main steps
of the abstract result of Theorem \ref{lem:suffexp} in the appendix. We begin by decomposing
the solution semigroup in a different way compared to the previous one in Section \ref{sub:dec}.
For $u_0\in \K_\alpha$, we split the solution as
$$
S_\alpha(t)u_0=V_\alpha(t)u_0+W_\alpha(t)u_0
$$
where
$$
v(t)=V_\alpha(t)u_0 \qquad \text{and}\qquad w(t)=W_\alpha(t)u_0
$$
respectively solve
\begin{equation}\label{eq:decay1}
\begin{cases}
\dot{v}+\alpha^2A\dot{v} +\nu Av=0,\\
v(0)=u_0.
\end{cases}
\end{equation}
and
\begin{equation}\label{eq:cpt1}
\begin{cases}
\dot{w}+\alpha^2A\dot{w} +\nu Aw+B(u,u)=g,\\
w(0)=0.
\end{cases}
\end{equation}
Notice that $V_\alpha(t)$ is a linear semigroup which is exponentially stable. It is then standard to see that
the following holds.
\begin{lemma}\label{lem:decay1}
Let $\alpha\in(0,1]$ and $u_0\in \K_\alpha$. Then
$$
\| V_\alpha(t)u_0-V_\alpha(t) v_0\|_{\V_\alpha}\leq \e^{-\frac{\kappa_\nu}{2} t} \|u_0-v_0\|_{\V_\alpha},
$$
for each $t\geq 0$.
\end{lemma}
Concerning $W_\alpha(t)$, all that is needed is a standard continuous dependence estimate.
\begin{lemma}\label{lem:cpt1}
The following inequality holds
$$
\| W_\alpha(t)u_0-W_\alpha(t)v_0\|_{\V_\alpha^2}\leq  \frac{c}{\alpha^{5/2}}\e^{\frac{c}{\alpha^2} (1+t)}\|u_0-v_0\|_{\V_\alpha},
$$
for each $t\geq 0$.
\end{lemma}

\begin{proof}
Let $w_1(t)=W_\alpha(t)u_0$ and $w_2(t)=W_\alpha(t)v_0$. The difference $w(t)=w_1(t)-w_2(t)$
satisfies
\begin{equation}\label{eq:diff1}
\begin{cases}
\dot{w}+\alpha^2A\dot{w} +\nu Aw+B(u_1,u)+B(u,u_2)=0,\\
w(0)=0,
\end{cases}
\end{equation}
where $u(t)=u_1(t)-u_2(t)=S_\alpha(t)u_0-S_\alpha(t)v_0$. Multiply the
above equation by $Aw$ in $\H$. The functional
$$
\Psi_\alpha(t)=\| W_\alpha(t)u_0\|_{\V_\alpha^2}^2=\| w(t)\|^2+\alpha^2|Aw(t)|^2
$$
satisfies the energy identity
$$
\frac12\ddt \Psi_\alpha +\nu |Aw|^2=-\l B(u_1,u),Aw\r-\l B(u,u_2),Aw\r.
$$
Thanks to Lemmas  \ref{lem:regest} and \ref{lem:pain1} and the Agmon inequality, we have
\begin{align*}
|\l B(u_1,u),Aw\r| &\leq c\|u_1\|_{L^\infty}\|u\||Aw|\leq c\|u_1\|^{1/2}|Au_1|^{1/2} \|u\||Aw|\\
&\leq \frac{\nu}{4}|Aw|^2+c\|u_1\||Au_1|\|u\|^2\leq \frac{\nu}{4}|Aw|^2+\frac{c}{\alpha^2}\|u_1\|_{\V_\alpha}\|u_1\|_{\V^2_\alpha}\|u\|^2\\
&\leq \frac{\nu}{4}|Aw|^2+c\frac{\sqrt{r_\alpha}}{\alpha^2}\|u\|^2\leq \frac{\nu}{4}|Aw|^2+\frac{c}{\alpha^5}\|u\|^2 .
\end{align*}
Similarly,
\begin{align*}
|\l B(u,u_2),Aw\r| &\leq c\|u\|\|u_2\|^{1/2}|Au_2|^{1/2}|Aw|\leq \frac{\nu}{4}|Aw|^2+c\|u_2\||Au_2|\|u\|^2\\
&\leq \frac{\nu}{4}|Aw|^2+\frac{c}{\alpha^2}\|u_2\|_{\V_\alpha}\|u_2\|_{\V^2_\alpha}\|u\|^2\leq \frac{\nu}{4}|Aw|^2+\frac{c}{\alpha^5}\|u\|^2.
\end{align*}
Thus, we end up with
$$
\ddt \Psi_\alpha \leq \frac{c}{\alpha^5}\|u\|^2\leq\frac{c}{\alpha^7}\|u\|_{\V_\alpha}^2.
$$
Thanks to the continuous dependence estimate \eqref{eq:contdep}, the above can be rewritten as
$$
\ddt \Psi_\alpha \leq \frac{c}{\alpha^7}\e^{\frac{c}{\alpha^2}(1+t)}\|u_0-v_0\|^2_{\V_\alpha}.
$$
Since $\Psi_\alpha(0)=0$, and integration over the time interval $(0,t)$ yields
$$
\Psi_\alpha(t) \leq \frac{c}{\alpha^5}\e^{\frac{c}{\alpha^2}(1+t)}\|u_0-v_0\|^2_{\V_\alpha},
$$
as we wanted.

\end{proof}

In order to prove the Lipschitz continuity of the solution map required by Theorem \ref{lem:suffexp},
we need a regularity estimate on the time derivative of trajectories originated from $\K_\alpha$. In fact,
since $\K_\alpha\subset \B_\alpha$, we will prove the following the following lemma for initial data
belonging to the absorbing set.

\begin{lemma}\label{lem:derivholder}
Let $u_0\in \B_\alpha$. We have
$$
\sup_{t\geq0}\|\dot u(t)\|_{\V_\alpha}\leq  \frac{c}{\alpha^{5/2}}.
$$
\end{lemma}

\begin{proof}
Multiply equation \eqref{eq:NSVABS} by $\dot u$, to get
\begin{equation}\label{eq:derest}
|\dot u|^2+\alpha^2\|\dot u\|^2=\l g,\dot u\r-\nu\l Au,\dot u\r-\l B(u,u),\dot u\r.
\end{equation}
On the one hand, we have
$$
|\l g,\dot u\r|\leq \frac12|\dot u|^2+ \frac12|g|^2
$$
and, in view of \eqref{eq:estind},
$$
\nu|\l Au,\dot u\r|\leq \frac{\alpha^2}{4}\|\dot u\|^2+ \frac{1}{\alpha^2}\|u\|^2\leq \frac{\alpha^2}{4}\|\dot u\|^2+ \frac{1}{\alpha^4}\|u\|_{\V_\alpha}^2
\leq \frac{\alpha^2}{4}\|\dot u\|^2+ \frac{c}{\alpha^4}.
$$
On the other hand, by \eqref{eq:b1}-\eqref{eq:b2} and again \eqref{eq:estind}, we obtain
\begin{align*}
|\l B(u,u),\dot u\r|&=|\l B(u,\dot u), u\r|\leq  c\|\dot u\||u|^{1/2}\|u\|^{3/2}\leq \frac{\alpha^2}{4}\|\dot u\|^2+\frac{c}{\alpha^2}|u|\|u\|^{3}\\
&\leq \frac{\alpha^2}{4}\|\dot u\|^2+\frac{c}{\alpha^5}\|u\|_{\V_\alpha}^4\leq \frac{\alpha^2}{4}\|\dot u\|^2+\frac{c}{\alpha^5}.
\end{align*}
Hence, \eqref{eq:derest} and the above estimates yield
$$
|\dot u|^2+\alpha^2\|\dot u\|^2\leq \frac{c}{\alpha^{5}},
$$
as we wanted.
\end{proof}

The immediate consequence of the previous lemma is the following result.

\begin{lemma}\label{lem:lipschitz}
Let $T>0$ be arbitrarily fixed. The map
$$
(t,u_0)\mapsto S_\alpha(t)u_0:[0,T]\times \K_\alpha\to \K_\alpha
$$
is Lipschitz continuous.
\end{lemma}

\begin{proof}
For $u_0,v_0\in \K_\alpha$ and $t_1,t_2\in [0,T]$ we have
$$
\|S_\alpha(t_1)u_0-S_\alpha(t_2)v_0\|_{\V_\alpha}\leq
\|S_\alpha(t_1)u_0-S_\alpha(t_1)v_0\|_{\V_\alpha}+\|S_\alpha(t_1)v_0-S_\alpha(t_2)v_0\|_{\V_\alpha}.
$$
The first term of the above inequality is handled by estimate \eqref{eq:contdep}. Concerning the second one,
Lemma \ref{lem:derivholder} implies
\begin{align*}
\|S_\alpha(t_1)v_0-S_\alpha(t_2)v_0\|_{\V_\alpha}&\leq \int_{t_1}^{t_2}\|\dot u(s)\|_{\V_\alpha}\d s\leq
c\frac{1}{\alpha^{5/2}}|t_1-t_2|.
\end{align*}
\end{proof}

Let us now conclude the proof of Theorem \ref{thm:expattr}, exploiting the abstract result
in Appendix \ref{app:exp}. It is clear that the set $\K_\alpha$ constructed above
has the required properties, since it is compact and invariant. Moreover, the Lipschitz continuity
devised in Lemma \ref{lem:lipschitz} takes care of the first assumption
in Theorem \ref{lem:suffexp}, no matter how $t_*$ will be chosen. We fix $t_*>0$ according to
Lemma \ref{lem:decay1}, so that
$$
 \e^{-\frac{\kappa_\nu}{2} t_*}=\frac18.
$$
In this way,
$$
\| V_\alpha(t_*)u_0-V_\alpha(t_*) v_0\|_{\V_\alpha}\leq\frac18 \|u_0-v_0\|_{\V_\alpha}, \qquad \forall t\geq0.
$$
Correspondingly, Lemma \ref{lem:cpt1} gives us
$$
\| W_\alpha(t_*)u_0-W_\alpha(t_*)v_0\|_{\V_\alpha^2}\leq C_*\|u_0-v_0\|_{\V_\alpha},
$$
with
\begin{equation}\label{eq:cstar}
C_*=\frac{c}{\alpha^{5/2}}\e^{\frac{c}{\alpha^2} (1+t_*)}\sim \frac{\e^{\alpha^{-2} }}{\alpha^{5/2}}.
\end{equation}
Therefore, we can apply Theorem \ref{lem:suffexp} with $V=V_\alpha(t_*)$ and $W=W_\alpha(t_*)$
and conclude the proof of Theorem \ref{thm:expattr}.

\subsection{Further comments}
We would like to conclude this section with some observations and possible further
development towards a better understanding of the longtime behavior of the Voigt
model.

\begin{remark}
In general, all the estimates devised for both global and exponential attractors depend explicitly
on $\alpha$, and, generically speaking, they are not robust as $\alpha$ vanishes.
This is due to the fact that, from the asymptotic behavior viewpoint, the limit problem
is far from being well understood. The asymptotic behavior of weak solutions to the
Navier-Stokes equations has attracted the efforts of many researchers over the years,
but, at the level of strong topologies, only partial results are available \cites{BA,CF,CTV,CTV2,KV,RS,RO,SE}. We will treat
this topic in the next sections, focusing on the relation between the Voigt model and the Navier-Stokes
equations.
\end{remark}

\begin{remark}
Another interesting research direction would lean towards a better understanding
of the robustness of the exponential attractors $\E_\alpha$. We expect the exponential
attractors to satisfy a H\"older continuity property with respect to the parameter $\alpha$,
at least on compact subintervals of $(0,1]$.
Specifically, for every $\alpha_0\in(0,1]$, we expect that there exist positive constants $C_0,\omega_0$ such that
$$
\sup_{\alpha_1,\alpha_2\in [\alpha_0,1]}
\max \big\{ \dist_{\V_{\alpha_1}}(\E_{\alpha_1},\E_{\alpha_2}),
\dist_{\V_{\alpha_1}}(\E_{\alpha_2},\E_{\alpha_1})\big\}
\leq C_0(\alpha_1-\alpha_2)^{\omega_0},
$$
with $\alpha_1>\alpha_2$. Again, a much harder problem would be to
allow $\alpha_0$ to be 0 in the above estimate, obtaining a result for the
Navier-Stokes equations as well.
\end{remark}

\section{Fractal dimension estimates}\label{sec:fractal}

\noindent Theorem \ref{thm:expattr} establishes the existence of exponential
attractors for $S_\alpha(t)$, for any $\alpha\in(0,1]$. As explicitly stated in
Theorem \ref{lem:suffexp}, the fractal dimension of $\E_\alpha$ can be estimated
in terms of $\alpha$ by knowing the number of $\V_\alpha$-balls of radius
$\eps_\alpha\sim \alpha^{5/2}\e^{-\alpha^{-2}}$ needed to cover the unit ball of $\V^2_\alpha$.
Indeed, this is the order of the constant $C_*$ from \eqref{eq:cstar} and Lemma \ref{lem:cpt1}.
In this case, the dimension estimate for the exponential attractor constructed by means of
Theorem \ref{lem:suffexp} will be exponential-like, and precisely
$$
\dim_{\V_\alpha}\,\A_\alpha\leq \dim_{\V_\alpha}\, \E_\alpha\sim \frac{\e^{1/\alpha^{2}}}{\alpha^{21/2}}.
$$
However, this approach ignores the differentiability of the semigroup $S_\alpha(t)$, which is
readily available from \cite{KT}*{Theorem 5.1}, and so it leads to non-optimal dimensional
estimates of the exponential attractor. An alternative construction of exponential attractors,
based on the method of Lyapunov exponents, is presented in \cites{Eden1, Eden2}. This approach
has the advantage that the dimension estimates for the exponential attractor, obtained through the
Lyapunov exponents, is shown to be optimal in terms of dimensionless parameters as in the case
of the 2D Navier-Stokes equations \cite{Eden1}*{Section 4}. For instance, for the latter, in
\cite{Eden1}*{Sections 4.1, 4.2} it was established the existence of an exponential attractor which
admits the same fractal dimension estimates as for the corresponding global attractor.

Our goal in this section is to establish an estimate for the fractal dimension
in $\V_\alpha$ for the global attractors $\A_\alpha$. This improves the best-known result in \cite{KT} by
producing a ``better" estimate as $\alpha\to 0$ (see final Remark \ref{important}). Using the above described approach,
an estimate in $\V_\alpha$ on the exponential attractors $\E_\alpha$ follows (see Theorem \ref{thm:exx}).

As a first step towards the fractal dimension estimate, we fix $\alpha\in (0,1]$ and
a solution $u$ to \eqref{eq:NSVABS} belonging to the global attractor $\A_\alpha$.
In this way, thanks to estimates \eqref{eq:absorb2}-\eqref{eq:estind}, we obtain that
\begin{equation}\label{eq:lin2}
\sup_{t\geq 0}\|u(t)\|_{\V_\alpha} \leq M_1=
\left[\frac{3(1+\lambda_1\alpha^2)|g|^2}{\nu^2\lambda_1^2}\right]^{1/2}, \quad \limsup_{T\to\infty}\frac{1}{T}\int_0^T \|u(t)\|^2 \d t\leq  K_1:=\frac{|g|^2}{\nu^2\lambda_1} .
\end{equation}
Setting
$$
G^2=I+\alpha^2 A
$$
and
$$
\hat{u}=Gu,
$$
we can rewrite the Navier-Stokes-Voigt equations \eqref{eq:NSVABS} as
\begin{equation}\label{eq:hat}
\begin{cases}
\displaystyle\dot{\hat{u}}=-\frac{\nu}{\alpha^2}\hat{u}+\frac{\nu}{\alpha^2}G^{-2}\hat{u}-G^{-1}B(G^{-1}\hat{u},G^{-1}\hat{u})+G^{-1}g,\\
\hat{u}(0)=\hat{u}_0=Gu_0.
\end{cases}
\end{equation}
Problem \eqref{eq:hat} is completely equivalent to the original Voigt model \eqref{eq:NSVABS}. However, it is well-posed in
$\H$, due to the fact that $G:\V_\alpha\to \H$ and $G^{-1}:\H\to\V_\alpha$ are isometries, namely
$$
|Gw|=\|w\|_{\V_\alpha}, \qquad |w|=\|G^{-1}w\|_{\V_\alpha}.
$$
More precisely, equations \eqref{eq:hat} generate a strongly continuous semigroup of solution operators
$$
\widehat{S}_\alpha (t):\H\to\H, \qquad \hat{u}_0\mapsto \hat{u}(t)=\widehat{S}_\alpha (t) \hat{u}_0,
$$
which are linked to the original semigroup $S_\alpha(t)$ through the relation
$$
\widehat{S}_\alpha (t)=G^{-1}S_\alpha(t)G.
$$
Concerning its longtime behavior, $\widehat{S}_\alpha(t)$ possesses the global attractor $\widehat{\A}_\alpha$, which can be seen to
satisfy
$$
\widehat{\A}_\alpha=G\A_\alpha.
$$
At first, it turns out to be more convenient to produce an estimate  on the fractal dimension in $\H$ for the attractor $\widehat{\A}_\alpha$.
As it will be clear in the next paragraph, an estimate in $\V_\alpha$ of the global attractor $\A_\alpha$ will follow in a straightforward manner.
The equation of linear variation corresponding to \eqref{eq:hat} reads
$$
\dot{w}=L(\hat{u};t)w
$$
where
\begin{equation}\label{eq:lin1}
L(\hat{u};t)w=-\frac{\nu}{\alpha^2}w+\frac{\nu}{\alpha^2} G^{-2}w-G^{-1} B(G^{-1}w,G^{-1}\hat{u})-G^{-1}B(G^{-1} \hat{u},G^{-1} w).
\end{equation}
The following proposition is crucial for our purposes.
\begin{proposition}\label{prop:lineariz}
Let $w\in \H$. Then
$$
\l L(\hat{u};t)w, w\r\leq-h_0|w|^2+h_1(t) |G^{-1}w|^2,
$$
where
$$
h_0=\frac{\nu}{2\alpha^2},\qquad h_1(t)=\frac{1}{\alpha^2}\left(\nu+c\frac{M_1^2}{\nu^3}\|u(t)\|^2\right) .
$$
\end{proposition}

\begin{proof}
By direct calculation, it is not hard to see that
\begin{equation}\label{eq:GGG}
\| G^{-1}w\|^2+\frac{1}{\alpha^2}|G^{-1}w|^2=\frac{1}{\alpha^2}|w|^2.
\end{equation}
Moreover,
\begin{equation}\label{eq:GGG2}
\l L(\hat{u};t)w, w\r=-\frac{\nu}{\alpha^2}|w|^2+\frac{\nu}{\alpha^2} |G^{-1}w|^2-\l  B(G^{-1}w,G^{-1}\hat{u}),G^{-1}w\r.
\end{equation}
We now estimate the trilinear term above as follows.  Noting that $G^{-1}\hat{u}=u$, we obtain
$$
|\l  B(G^{-1}w,G^{-1}\hat{u}),G^{-1}w\r| \leq c|G^{-1}w|^{1/2}\|G^{-1}w\|^{3/2} \|u\|,
$$
and by means of $\eps$-Young's inequality and \eqref{eq:GGG}, there holds
\begin{align*}
|\l  B(G^{-1}w,G^{-1}\hat{u}),G^{-1}w\r| &\leq c |G^{-1}w|^{1/2}\|G^{-1}w\|^{3/2} \|u\| \\
&\leq\frac{c}{\alpha^{3/2}} |G^{-1} w|^{1/2}|w|^{3/2}\|u\|\\
&\leq \frac{3}{4} \eps |w|^2 +\frac{c}{\alpha^6} |G^{-1}w|^2\|u\|^4 \frac{1}{4\eps^3}.
\end{align*}
Setting $\eps>0$ as $\eps= \frac{2\nu}{3\alpha^2}$, we find that
$$
|\l  B(G^{-1}w,G^{-1}\hat{u}),G^{-1}w\r| \leq \frac{\nu}{2\alpha^2} |w|^2 +\frac{c}{\nu^3} \|u\|^4|G^{-1}w|^2.
$$
In light of the first inequality in \eqref{eq:lin2} and
$$
 \|u\|^4=\|u\|^2\|u\|^2\leq \frac{1}{\alpha^2}\|u\|_{\V_\alpha}^2\|u\|^2\leq \frac{M_1^2}{\alpha^2}\|u\|^2,
$$
we end up with the estimate
\begin{equation}\label{eq:GGG3}
|\l  B(G^{-1}w,G^{-1}\hat{u}),G^{-1}w\r| \leq\frac{\nu}{2\alpha^2} |w|^2 +c\frac{M_1^2}{\alpha^2\nu^3} \|u\|^2|G^{-1}w|^2.
\end{equation}
Going back to \eqref{eq:GGG2}, this implies that
$$
\l L(\hat{u};t)w, w\r\leq-\frac{\nu}{2\alpha^2}|w|^2+\frac{1}{\alpha^2}\left(\nu+c\frac{M_1^2}{\nu^3}\|u\|^2\right) |G^{-1}w|^2,
$$
which is what we wanted to prove.
\end{proof}

The main result of this section is the following.
\begin{theorem}\label{thm:fractunif}
For every $\alpha\in (0,1]$, the global attractors $\A_\alpha$ have finite fractal dimension in $\V_\alpha$. Precisely,
\begin{equation}\label{eq:lin3}
\dim_{\V_\alpha}\, \A_\alpha \leq \frac{c}{\alpha^3}\left[1+\frac{1+\lambda_1\alpha^2}{\nu^8\lambda_1^4}|g|^4\right]^{3/2},
\end{equation}
where $c>0$ is a dimensionless scale invariant constant.
\end{theorem}
It is a well-know fact that fractal dimension estimates are preserved by Lipschitz maps (see e.g.
\cite{T3}*{Proposition 3.1}) and since $G$ and $G^{-1}$ are isometries, in our particular case we have
$$
\dim_{\V_\alpha}\, \A_\alpha=\dim_{\V_\alpha}\, G^{-1}\widehat{\A}_\alpha= \dim_\H\, \widehat{\A}_\alpha.
$$
Consequently, Theorem \ref{thm:fractunif} follows from the fact that $\widehat{\A}_\alpha$ has finite fractal dimension in $\H$, with the same bound \eqref{eq:lin3},
as proven in the following proposition.

\begin{proposition}\label{prop:fractunif}
For every $\alpha\in (0,1]$, the global attractors $\widehat{\A}_\alpha$ have finite fractal dimension in $\H$, with
\begin{equation}\label{eq:lin4}
\dim_{\H}\, \widehat{\A}_\alpha \leq \frac{c}{\alpha^3}\left[1+\frac{1+\lambda_1\alpha^2}{\nu^8\lambda_1^4}|g|^4\right]^{3/2},
\end{equation}
where $c>0$ is a dimensionless scale invariant constant.
\end{proposition}
\begin{proof}
Consider an initial orthogonal set of infinitesimal displacements $w_{1,0},\ldots,w_{n,0}$, for some $n\geq 1$.
The volume of the parallelepiped they span is given by
$$
V_n(0)=|w_{1,0}\wedge\ldots \wedge w_{n,0}|.
$$
It follows that the evolution of such displacements obeys the evolution equation
$$
\dot{w}_i=L(\hat{u};t)w_i, \qquad w_i(0)=w_{i,0}, \qquad \forall i=1,\ldots,n.
$$
Then it follows cf. \cites{CONFO, COST} that the volume elements
$$
V_n(t)=|w_{1}(t)\wedge\ldots \wedge w_{n}(t)|
$$
satisfy
$$
V_n(t)=V_n(0)\exp\left[\int_0^t \mathrm{Tr}(P_n(s)L(\hat{u};s) )\d s\right],
$$
where the orthogonal projection $P_n(s)$ is onto the linear span of $\{w_{1}(s),\ldots, w_{n}(s)\}$ in $\H$,
and
$$
\mathrm{Tr}(P_n(s)L(\hat{u};s))=\sum_{j=1}^n \l L(\hat{u};s) \phi_j(s),\phi_j(s)\r,
$$
with $n\geq 1$ and $\{\phi_1(s),\ldots,\phi_n(s)\}$ an orthonormal set spanning $P_n(s)\H$.
Letting
$$
\l\l P_nL(\hat{u})\r\r:=\limsup_{T\to\infty} \frac1T\int_0^T \mathrm{Tr}(P_n(t)L(\hat{u};t))\d t,
$$
we obtain
$$
V_n(t)\leq V_n(0)\exp\left[ t\sup_{\hat{u}\in \widehat{\A}_\alpha}\sup_{P_n(0)}\l\l P_nL(\hat{u})\r\r\right]
$$
for all $t\geq 0$, where the supremum over $P_n(0)$ is a supremum over all choices of initial $n$ orthogonal set of
infinitesimal displacements that we take around $\hat{u}$.
We then need to show that $V_n(t)$ decays exponentially in time whenever $n\geq N$, with $N>0$ to be determined.
To achieve this, we will make use of the estimate
\begin{equation}\label{eq:lowerstokes}
\lambda_j\geq c\lambda_1 j^{2/3}, \qquad \forall j\geq 1, 
\end{equation}
on the eigenvalues of the Stokes operator in three dimensions, derived in \cites{I1,I2} (see also \cites{COST,I3,MET}
for an asymptotic estimate with various boundary conditions).

Thanks to Proposition \ref{prop:lineariz}, Lemma 6.2 (pp. 454) of \cite{T3} and \eqref{eq:lowerstokes} 
we have
\begin{align*}
\frac1T\int_0^T \mathrm{Tr}(P_n(t)L(\hat{u};t))\d t&= \frac1T\int_0^T \sum_{j=1}^n \l L(\hat{u};t) \phi_j(t),\phi_j(t)\r\d t\\
&\leq\frac1T\int_0^T \sum_{j=1}^n -h_0|\phi_j(t)|^2\d t+\frac1T  \int_0^Th_1(t)  \sum_{j=1}^n |G^{-1}\phi_j(t)|^2\d t\\
&\leq -h_0 n+ \frac1T  \int_0^T h_1(t)\d t  \sum_{j=1}^n \frac{1}{1+\lambda_1 \alpha^2 j^{2/3}}\\
&\leq -h_0 n + c \frac{n^{1/3}}{\lambda_1 \alpha^2}\frac1T  \int_0^T h_1(t)\d t
\end{align*}
In view of the bounds on the global attractor given by \eqref{eq:lin2} and Proposition  \ref{prop:lineariz},
it is now straightforward to check that
$$
\l\l P_nL(\hat{u})\r\r\leq   -\frac{\nu}{2\alpha^2}n + \frac{1}{\alpha^4}\left(\nu+c\frac{M_1^2K_1}{\lambda_1 \nu^3}\right)n^{1/3}.
$$
In order to have the above quantity negative, we need to require
$$
n\geq N:= \frac{c}{\alpha^3}\left[1+\frac{M_1^2K_1}{\lambda_1 \nu^4}\right]^{3/2}.
$$
Recalling the definitions of the constants $M_1$ and $K_1$, we get
$$
N= \frac{c}{\alpha^3}\left[1+\frac{1+\lambda_1\alpha^2}{\nu^8\lambda_1^4}|g|^4\right]^{3/2},
$$
concluding the proof of the theorem.
\end{proof}
In terms of the three-dimensional Grashof number
$$
\mathfrak{G}=\frac{|g|}{\nu^2 \lambda_1^{3/4}},
$$
the fractal dimension estimate can be written as
\begin{equation}\label{eq:grashof}
\dim_{\V_\alpha}\, \A_\alpha \leq \frac{c}{\alpha^3}\left[1+\frac{1+\lambda_1\alpha^2}{\lambda_1} \mathfrak{G}^4\right]^{3/2}.
\end{equation}

The optimality of the above estimate in terms of $\mathfrak{G}$ is an open question. Nonetheless,
\eqref{eq:grashof} gives some further insights on the Navier-Stokes equations (as the limit problem as $\alpha\to 0$), which
we discuss in the remark below.
\begin{remark}
In various works \cites{CONFO,CFMT,CFT,FMRT,GT}, several estimates on the dimension of global
attractor of the three-dimensional Navier-Stokes equations were derived, although it is still an outstanding
open problem to show that it exists. In particular, it has been shown that if $\A\subset  \H$ is invariant and bounded
in $\V$, then
$$
\dim_{\H}\, \A \leq c \mathfrak{G}^{3/2}.
$$
The above bound has been derived heuristically via dimensional analysis, and it is not known that such an attractor $\A$
actually even exists for the Navier-Stokes equations. In particular, the only available result is the existence of a
\emph{weak} attractor (see next section), which is only known to be bounded in $\H$.
On the other hand, in \cite{LIU}, a lower bound of the form
$$
\dim_{\H}\, \A \geq c \mathfrak{G}
$$
was obtained.
\end{remark}

\begin{remark}\label{important}
An upper bound for the fractal dimension of the global attractor $\A_\alpha$, of the order $c\alpha^{-6}$ as $\alpha\to 0$, was obtained in \cite{KT} employing a result that can be found in the book of Ladyzhenskaya, see Theorem 4.9 of \cite{LADY}. Unfortunately, the claim of this result remains generally untrue due to a faulty assumption (see, the second condition of (4.36) in \cite{LADY} on pg. 32) which is generally false for a given positive self-adjoint operator. In particular, such an assumption does not hold for the Stokes operator $A$. This is the reason why we have chosen to work directly with the Constantin-Foias trace formula.
\end{remark}

We conclude this section with the following result on exponential attractors associated with the Voigt model.
\begin{theorem}\label{thm:exx}
The semigroup $S_\alpha(t)$ admits an exponential attractor $\E_\alpha$ whose fractal dimension obeys the estimate:
$$
\dim_{\V_\alpha}\,\E_\alpha\leq 1+\dim_{\V_\alpha}\, \A_\alpha.
$$
\end{theorem}
\begin{proof}
By virtue of Lemmas \ref{lem:decay1} and \ref{lem:cpt1}, we immediately see that the
semigroup $S_\alpha(t)$ is a contraction in the sense of \cite{Eden2}*{Definition 2.1}.
Furthermore, by Lemma \ref{lem:lipschitz} it is also Lipschitz. The claim follows then from
the application of \cite{Eden2}*{Theorem 2.2}.
\end{proof}

\section{On the Navier-Stokes limit of the Voigt model}
\noindent In this section, we establish various results on the convergence of the (strong) global attractors for the Voigt
model as $\alpha$ goes to zero. In contrast to the trajectory dynamical approach employed in \cites{CTV,CTV2} for
the 3D Navier-Stokes $\alpha$-model and the 3D Leray $\alpha$-model, here we only use the (simple) concept of
multivalued semiflows and weak topology of $L^2$, by choosing to work directly in the physical phase space $\H$.

\subsection{Leray-Hopf weak solutions}\label{sub:LH}
A Leray-Hopf weak solution to \eqref{eq:NSABS} is a function $u:[0,\infty)\to \H$ such that
\begin{enumerate}[label=(\roman*)]
	\item $u\in L^\infty_{loc}(0,\infty;\H)\cap L^2_{loc}(0,\infty;\V)$;
	\item $\dot{u}\in L^{4/3}_{loc}(0,\infty;\V^*)$;
	\item $u\in C([0,\infty);\H_w)$, namely, for every $v\in \H$, the function
	$t\mapsto \l u(t),v\r$ is continuous from $[0,\infty)$ into $\R$;
	\item for every $s,t\in [0,\infty)$ with $s\leq t$, there holds
	$$
	\l u(t),v\r + \nu\int_s^t\l Au(\tau),v \r\d\tau+\int_s^t\l B(u(\tau),u(\tau)),v \r\d\tau=\l u(s),v\r+
	 \int_s^t\l g,v \r\d\tau, \qquad \forall v\in\V;
	$$
	\begin{equation}\label{eq:enerineq}
	\frac12\frac{\d}{\d t}|u(t)|^2+\nu\|u(t)\|^2\leq \l g,u(t)\r, \qquad t\in[0,T],
	\end{equation}
	in the distribution sense.
\end{enumerate}
The classical results of Leray \cite{L} and Hopf \cite{HO} establish in particular that
for any given $u_0\in \H$, there exists at least a solution to \eqref{eq:NSABS}
in the above sense. However, the uniqueness of such solutions is still unknown.

\begin{remark}
Inequality \eqref{eq:enerineq} should be understood in the following sense: for each
positive function $\psi\in C_0^\infty((0,T))$,
\begin{equation}\label{eq:enerineq1}
-\frac12\int_0^T|u(t)|^2\psi^\prime(t)\d t+\nu\int_0^T\|u(t)\|^2\psi(t)\d t\leq \int_0^T\l g,u(t)\r\psi(t)\d t.
\end{equation}
It is worth observing here that any solution constructed by means of a Faedo-Galerkin
approximation procedure satisfies \eqref{eq:enerineq}, but this is not known for
weak solutions in general.
\end{remark}

\begin{remark}
Inequality \eqref{eq:enerineq} is sometimes rephrased as follows:  for almost every $s\in [0,\infty)$, $u$ satisfies the energy inequality
	\begin{equation}\label{eq:enerineq2}
	|u(t)|^2+2\nu\int_s^t\| u(\tau)\|^2\d \tau\leq |u(s)|^2+2\int_s^t \l g,u(\tau)\r \d\tau,
	\end{equation}
for all $t\geq s$.
The set (of full measure) of times $s$ for which the energy inequality \eqref{eq:enerineq2} holds
coincides with the points of strong continuity from the right of $u$ in $\H$.
It is important to notice that we do not require  the energy inequality \eqref{eq:enerineq2} to hold
at $s=0$. In this way, we are guaranteed that translations of solutions are still solutions, while
the concatenation property is unknown. Our approach slightly differs
from \cites{FRT,FRT2,RO,SE}, while, in this sense at least, it is more closely related to that in \cite{CF}.
\end{remark}

For every $t\geq 0$, we  define a multivalued function $S(t):\H\to 2^\H$ by
$$
S(t)u_0=\big\{u(t):  u \text{ is a solution to \eqref{eq:NSABS} with } u(0)=u_0\big\}.
$$
It is clear that $S(0)$ is the identity on $\H$, while the fact that translations of solutions are still solutions
implies the inclusion
\begin{equation}\label{eq:semtra}
S(t+\tau)u_0\subset S(t)S(\tau)u_0, \qquad \forall t,\tau\geq 0, \quad \forall u_0\in\H.
\end{equation}
In the literature, the above objects are referred to as \emph{multivalued semiflows}, and were introduced
for the first time in \cite{MV}. Different approaches to the study of the asymptotic behavior of systems
without uniqueness properties of solutions were developed in  \cites{BA, CV, CHES}.
It is well known (see \cite{COST}) that the set
$$
\B_0=\left\{w\in \H: |w|^2\leq\frac{2(1+\lambda_1)|g|^2}{\nu^2\lambda_1^2}\right\}
$$
is a bounded absorbing set for $S(t)$.

%

\begin{remark}
In fact, any radius bigger than $|g|/\nu\lambda_1$ will produce
an absorbing ball. The reason why we chose this particular one is that we will be able
to compare it with the absorbing set of the Voigt model, defined in \eqref{eq:absball}. Notice that
$\B_\alpha\subset \B_0$ for every $\alpha\in(0,1]$.
\end{remark}

In what follows, we will study the asymptotic behavior of the Navier-Stokes equations
with respect to two different metrics.

\begin{itemize}
\item Since $\H$ is a separable Hilbert space, the weak
topology on bounded sets is metrizable, by a metric that we will denote by $\d_w$.
The classical approach first devised in \cite{FT} consists, roughly speaking, in studying
the asymptotic behavior of the absorbing ball $\B_0$ with respect to the weak metric $\d_w$,
in terms of the so-called weak attractor.

\item The harder (and, for the moment, out of reach) case is the one involving the strong
topology of $\H$. It is perhaps one of the most challenging problems in the theory of infinite-dimensional
dynamical system to prove that the above-mentioned weak attractor is strong.

\end{itemize}

\subsection{Weak attractors}
Although formulated in different ways, one of the main results of \cite{CF} and \cite{RO} can be stated as follows.

\begin{theorem}\label{thm:weakstrong}
There exists a weakly compact set $\A_w\subset \B_0$ such that
$$
\lim_{t\to\infty}\dist_w(S(t)\B_0,\A_w)=0,
$$
where $\dist_w$ stands for the Hausdorff semidistance between sets given by the
weak metric $\d_w$. Moreover, $\A_w$ is unique, it is the weak $\omega$-limit set of $\B_0$ and
$$
\A_w=\left\{u(0): u \text{ is a solution on } (-\infty,\infty) \text{ with } u(t)\in\B_0,\ \forall t\in\R\right\}.
$$
Also,  the following statements are equivalent:
\begin{enumerate}[label=(\roman*)]
	\item $\A_w$ is strongly compact in $\H$;
	\item $\A_w$ is the strong attractor of $S(t)$;
	\item  all solutions on the weak global attractor are strongly continuous in $\H$.
\end{enumerate}
\end{theorem}
The existence of the weak attractor of the 3D NSE dates back to the seminal work of Foias and Temam \cite{FT},
while one of the novelties of  \cite{CF} and \cite{RO} is the necessary and sufficient condition of
strong continuity ensuring the existence of a strong attractor. In this spirit, we shall investigate
what kind of information can be deduced about the weak attractor of the NSE from the Voigt regularization.

%
%

\subsection{Convergence as $\alpha\to 0$}
At the level of weak topologies, it can be rigorously shown that the global attractors $\{\A_\alpha\}_{\alpha\in(0,1]}$
are upper semicontinuous with respect to $\alpha$.
The main result of this section reads as follows.
\begin{theorem}\label{thm:mainweak}
The family of attractors $\{\A_\alpha\}_{\alpha\in(0,1]}$ of the Voigt equations converges, as $\alpha\to 0$, to the weak
attractor $\A_w$ of the Navier-Stokes equations, namely
\begin{equation}\label{eq:convVstar}
\lim_{\alpha\to 0}\dist_{w}(\A_\alpha,\A_w)=0.
\end{equation}
\end{theorem}
Theorem \ref{thm:mainweak} rephrases in a very simple way some of the results
obtained in \cites{CTV,CTV2} for different models. In these papers, the
authors make use of the theory of trajectory attractors in order to deal with the
possible non-uniqueness of Leray-Hopf weak solutions, obtaining convergence
results analogous to \eqref{eq:convVstar} in the so-called trajectory spaces and their
(non-metrizable) topologies.

We start by establishing some standard estimates for the Voigt problem which are
uniform with respect to $\alpha$.

\begin{lemma}\label{lem:estind}
Let $u_0\in \B_\alpha$ and set $u^\alpha(t)=S_\alpha(t)u_0$. Then
\begin{equation}\label{eq:estind1}
\sup_{t\geq 0} \left[|u^\alpha(t)|^2+\alpha^2\|u^\alpha(t)\|^2+\nu\int_{t}^{t+1}\|u^\alpha(\tau)\|^2\d \tau+\int_{t}^{t+1}\|\dot{u}^\alpha(\tau)\|_*^{4/3}\d \tau\right]\leq c,
\end{equation}
where $c$ is independent of $\alpha$.
\end{lemma}

\begin{proof}
The first part of the above estimate is a straightforward consequence of \eqref{eq:estind}. Now, integrating \eqref{eq:firstdiff}
on $(t,t+1)$, we find
$$
|u^\alpha(t+1)|^2+\alpha^2\|u^\alpha(t+1)\|^2 +\nu\int_t^{t+1}\|u^\alpha(\tau)\|^2\d\tau\leq \frac{|g|^2}{\nu\lambda_1}+|u^\alpha(t)|^2+\alpha^2\|u^\alpha(t)\|^2,
$$
so that, a further use of \eqref{eq:estind} entails
$$
\nu\int_t^{t+1}\|u^\alpha(\tau)\|^2\d\tau\leq c.
$$
For the estimate on $\dot{u}^\alpha$, we rewrite the Voigt model as
$$
\dot{u}^\alpha +\nu (I+\alpha^2A)^{-1} Au^\alpha+(I+\alpha^2A)^{-1}B(u^\alpha,u^\alpha)=(I+\alpha^2A)^{-1}g.
$$
Therefore,
$$
\|\dot{u}^\alpha\|_*\leq \nu \|(I+\alpha^2A)^{-1} Au^\alpha\|_*+\|(I+\alpha^2A)^{-1}B(u^\alpha,u^\alpha)\|_*+\|(I+\alpha^2A)^{-1}g\|_*.
$$
As a consequence,
$$
\|\dot{u}^\alpha\|_*\leq c \left[\nu \|u^\alpha\|+\|B(u^\alpha,u^\alpha)\|_*+\|g\|_*\right].
$$
By standard estimates, it is not hard to see that
$$
\|B(u^\alpha,u^\alpha)\|_*\leq c \|u^\alpha\|^{3/2}|u^\alpha|^{1/2}.
$$
Hence, in light the above estimates, we learn that
$$
\int_{t}^{t+1}\|\dot{u}^\alpha(\tau)\|^{4/3}_*\d\tau\leq c \left[\int_{t}^{t+1}\|u^\alpha(\tau)\|^2\d\tau+\|g\|^{4/3}_*\right]\leq c,
$$
and the proof is over.
\end{proof}

With the above lemma at our disposal, we then can prove the following convergence result, which is crucial for
our purposes.

\begin{lemma}\label{lem:conv}
Let $\{\alpha_n\}_{n\in\N}\subset (0,1]$ be any sequence such that $\alpha_n\to 0$ as $n\to \infty$. Assume
that, for each $n\in \N$, $u_{0,n}\in \B_{\alpha_n}$. There exists $u_0\in \B_0$, a Leray-Hopf
weak solution $u:[0,\infty)\to \H$ to \eqref{eq:NSABS} with $u(0)=u_0$ and a subsequence $\{n_k\}_{k\in\N}$
such that
\begin{equation}\label{eq:conv1}
\lim_{k\to\infty}\d_w(S_{\alpha_{n_k}}(t)u_{0,{n_k}}, u(t))=0, \qquad \forall t\geq0.
\end{equation}
Consequently, for each $t\geq 0$,
\begin{equation}\label{eq:conv2}
\lim_{k\to\infty}\dist_w(S_{\alpha_{n_k}}(t)u_{0,{n_k}}, S(t)u_0)=0.
\end{equation}
\end{lemma}

\begin{proof}
Thanks to the weak compactness of $\B_0$ and the inclusion
$$
\B_{\alpha_n}\subset \B_0, \qquad \forall n\in\N,
$$
we can find a limit point $u_0\in \B_0$ such that, up to not relabeled subsequences,
$$
u_{0,n}\rightharpoonup u_0, \qquad \text{weakly in }\H.
$$
Moreover,
$$
\alpha_n^2 u_{0,n}\to 0, \qquad \text{strongly in }\V.
$$
Thanks to Lemma \ref{lem:estind}, we realize at once that, for every $T>0$, the sequence
of solutions $u_n(t)=S_{\alpha_n}(t)u_{0,n}$ fulfill the uniform bounds
\begin{align*}
&\{u_n\}_{n\in\N}\subset \text{bdd set of } L^\infty(0,T;\H)\cap L^2(0,T,\V),\\
&\{\dot{u}_n\}_{n\in\N}\subset \text{bdd set of } L^{4/3}(0,T,\V^*).
\end{align*}
The  Aubin-Lions lemma then ensures the existence of a function $u:[0,\infty)\to \H$ such that,
for each $T>0$,
\begin{align*}
&u\in  L^\infty(0,T;\H)\cap L^2(0,T,\V)\cap C([0,T],\H_w),\\
&\dot{u}\in  L^{4/3}(0,T,\V^*),
\end{align*}
and such that the following convergences hold (again, up to subsequences):
\begin{align*}
&u_n\rightharpoonup  u \qquad \text{weakly in }     L^2(0,T,\V),\\
&u_n\to u\qquad \text{strongly in }     L^2(0,T,\H).
\end{align*}
Notice that the above implies also weak pointwise convergence for \emph{every} $t\in [0,T]$, due to
the (weak) continuity of the limit. Precisely, for every $v\in \H$ and every $t\in [0,T]$,
\begin{equation}\label{eq:weakl2}
\l u_n(t),v\r \to \l u(t),v\r, \qquad \text{as } n\to \infty,
\end{equation}
and in particular we find $u(0)=u_0$ as well. Finally, thanks to \eqref{eq:estind1} one more time, as $n\to\infty$ we have
\begin{equation}\label{eq:goes0}
\sup_{t\in[0,T]}\alpha_n^2\|u_n(t)\| \to 0, \qquad \forall t\in[0,T],
\end{equation}
and
\begin{equation}\label{eq:goes01}
\alpha_n^2\int_0^T\|u_n(t)\|^2\d t \to 0.
\end{equation}
Our next goal is to show that $u$ satisfies
the Navier-Stokes equations in the weak sense specified in Paragraph \ref{sub:LH}. For each
$n\in \N$ and every $s,t\in [0,T]$ with $s\leq t$, there holds
\begin{align*}
	&\l u_n(t),v\r +\alpha_n^2\l A^{1/2}u_n(t),A^{1/2}v\r+ \nu\int_s^t\l A^{1/2}u_n(\tau),A^{1/2}v \r\d\tau+\int_s^t\l B(u_n(\tau),u_n(\tau)),v \r\d\tau\\
&\quad=\l u_n(s),v\r +\alpha_n^2\l A^{1/2}u_n(s),A^{1/2}v\r+ \int_s^t\l g,v \r\d\tau, \qquad \forall v\in\V;
	\end{align*}
Each term in the above equation converges, and in particular, in view of \eqref{eq:goes0}, we have
$$
\lim_{n\to\infty}\alpha_n^2\l A^{1/2}u_n(t),A^{1/2}v\r=0, \qquad \forall t\in[0,T].
$$
Therefore, we find that
\begin{align*}
\l u(t),v\r  +\nu\int_s^t\l A^{1/2}u(\tau),A^{1/2}v \r\d\tau+\int_s^t\l B(u(\tau),u(\tau)),v \r\d\tau
=\l u(s),v\r + \int_s^t\l g,v \r\d\tau,
\end{align*}
which implies that $u$ is a weak solution to the Navier-Stokes equations. It now remains
to show that $u$ satisfies the energy inequality \eqref{eq:enerineq}, which we want to
verify in the sense of distribution, as explained in \eqref{eq:enerineq1}. To this end,
we first observe that, for every $n\in\N$, the function $u_n$
satisfies the energy equation
\begin{equation}\label{eq:energyeqw}
-\frac12\int_0^T\big[|u_n(t)|^2+\alpha_n^2\|u_n(t)\|^2\big]\psi^\prime(t)\d t+\nu\int_0^T\|u_n(t)\|^2\psi(t)\d t= \int_0^T\l g,u_n(t)\r\psi(t)\d t,
\end{equation}
for every $T>0$ and every positive function $\psi\in C^\infty_0((0,T))$.
Possibly passing to a further subsequence, we can assume that
$$
|u_n(t)|^2\to |u(t)|^2 \qquad \text{as }n\to\infty,\quad \text{for almost all }t\in[0,T],
$$
and, in view of \eqref{eq:goes01}, also that
$$
\alpha_n^2\|u_n(t)\|^2\to 0 \qquad \text{as }n\to\infty,\quad \text{for almost all }t\in[0,T].
$$
The sequence of real-valued functions $\big\{(|u_n(\cdot)|^2+\alpha_n^2\|u_n(\cdot)\|^2)\psi^\prime\big\}$ is in $L^1(0,T)$
and it is essentially bounded. Therefore, by the Lebesgue dominated convergence theorem, we have
$$
\lim_{n\to\infty}\int_0^T\big[|u_n(t)|^2+\alpha_n^2\|u_n(t)\|^2\big]\psi^\prime(t)\d t=\int_0^T|u(t)|^2\psi^\prime(t)\d t.
$$
Moreover, observing that $u_n\sqrt{\psi}$ converges weakly in $L^2(0,T;\V)$ to $u\sqrt{\psi}$, there holds
$$
\int_0^T\|u(t)\|^2\psi(t)\d t\leq \liminf_{n\to\infty}\int_0^T\|u_n(t)\|^2\psi(t)\d t.
$$
Finally, it is clear that
$$
\lim_{n\to\infty}\int_0^T\l g,u_n(t)\r\psi(t)\d t=\int_0^T\l g,u(t)\r\psi(t)\d t.
$$
Taking the ``$\liminf$'' of both sides in the energy equation \eqref{eq:energyeqw}, we find the
corresponding energy inequality for $u$, namely
$$
-\frac12\int_0^T|u(t)|^2\psi^\prime(t)\d t+\nu\int_0^T\|u(t)\|^2\psi(t)\d t\leq \int_0^T\l g,u(t)\r\psi(t)\d t,
$$
concluding the proof.
\end{proof}

Consequently, we can prove

\begin{corollary}\label{cor:conv}
For every $t\geq 0$, there holds
\begin{equation}\label{eq:conv3}
\lim_{\alpha\to0}\dist_w(S_{\alpha}(t)\B_{\alpha}, S(t)\B_0)=0.
\end{equation}
\end{corollary}

\begin{proof}
If not, there exists $t\geq 0$, $\eps>0$ and sequences $\alpha_n\in(0,1]$ with $\alpha_n\to 0$ as $n\to\infty$,
$u_{0,n}\in \B_{\alpha_n}$ such that,
$$
\inf_{v\in S(t)\B_0} \dist_w(S_{\alpha_n}(t)u_{0,n},v)\geq \eps, \qquad \forall n\in\N.
$$
In particular,  for every $v\in S(t)\B_0$, we necessarily have that
$$
\d_w(S_{\alpha_n}(t)u_{0,n},v)\geq \eps, \qquad \forall n\in\N.
$$
However, this is contradicted by Lemma \ref{lem:conv}.
\end{proof}
We conclude this section the proof of Theorem \ref{thm:mainweak}.

\begin{proof}[Proof of Theorem \ref{thm:mainweak}]
By the triangle inequality, we have
$$
\dist_w(\A_\alpha,\A_w)\leq \dist_w(\A_\alpha, S(t)\B_0)+\dist_w(S(t)\B_0,\A_w),
$$
for every   $\alpha \in (0,1]$ and $t>0$.
Let $\eps>0$ be arbitrarily fixed. Since $\A_w$ is the weak attractor of $S(t)$, there exists
$t_\eps>0$ such that
$$
\dist_w(S(t_\eps)\B_0,\A_w)\leq \frac{\eps}{2}.
$$
Also, since $\A_\alpha$ is invariant under $S_\alpha(t)$ and is contained in $\B_\alpha$, we have
$$
\dist_w(\A_\alpha, S(t_\eps)\B_0)=\dist_w(S_\alpha(t_\eps)\A_\alpha, S(t_\eps)\B_0)\leq \dist_w(S_\alpha(t_\eps)\B_\alpha, S(t_\eps)\B_0).
$$
Now, Corollary \ref{cor:conv} ensures the existence of some $\alpha_\eps\in (0,1]$ such that
$$
\dist_w(S_\alpha(t_\eps)\B_\alpha, S(t_\eps)\B_0)\leq \frac{\eps}{2}, \qquad \forall \alpha\leq \alpha_\eps.
$$
Thus,
$$
\dist_w(\A_\alpha,\A_w)\leq \dist_w(S_\alpha(t_\eps)\B_\alpha, S(t_\eps)\B_0)+\dist_w(S(t_\eps)\B_0,\A_w)\leq \eps, \qquad \forall \alpha\leq \alpha_\eps.
$$
Since $\eps$ was arbitrary to begin with, we can conclude that
$$
\lim_{\alpha\to0}\dist_w(\A_\alpha,\A_w)=0,
$$
as wanted.
\end{proof}

\begin{remark}
The result of Theorem \ref{thm:mainweak} can be rephrased in terms of the Hausdorff semidistance in $\V^*$ as
$$
\lim_{\alpha\to 0}\dist_{\V^*}(\A_\alpha,\A_w)=0,
$$
and even in $\V^0_\alpha$, by writing
$$
\lim_{\alpha\to 0}\dist_{\V^0_\alpha}(\A_\alpha,\A_w)=0,
$$
where $\V^0_\alpha$ is the space $\H$ endowed with the $\alpha$-dependent norm (see Section \ref{sub:mathset})
$$
\|u\|^2_{\V_\alpha^0}= \|u\|_*^2+\alpha^2|u|^2.
$$
This is a consequence of the well-known fact that $\H$-bounded sequences converge weakly if and only if they
converge strongly in $\V^*$. The use of parameter-dependent spaces has been particularly useful in the context
of singular hyperbolic-parabolic limits and their dependencies on the perturbation parameter,
see \cites{EMZ,EZM,GGMP,GMPZ,MPZ,FGMZ} and references therein.
\end{remark}

\subsection{Strong attractors}
We address here the question of convergence of the attractors $\A_\alpha$ to $\A_w$ in the strong topology of $\H$.
In particular, we derive a necessary condition for $\A_w$ to be strong in terms of the upper semicontinuity of
$\A_\alpha$ as $\alpha\to 0$, while a sufficient condition is obtained in terms of the symmetric Hausdorff distance in $L^2$.
The arguments resemble those of the previous paragraph, and only the major changes will be highlighted.
To begin, we need a ``strong'' version of Lemma \ref{lem:conv}.
\begin{lemma}\label{lem:strongconv}
Assume that $\A_w$ is the strong attractor of $S(t)$.
Let $\{\alpha_n\}_{n\in\N}\subset (0,1]$ be any sequence such that $\alpha_n\to 0$ as $n\to \infty$. Assume
that, for each $n\in \N$, $u_{0,n}\in \A_{\alpha_n}$. There exists $u_0\in \B_0$, a global Leray-Hopf
weak solution $u:(-\infty,\infty)\to \H$ to \eqref{eq:NSABS} with $u(0)=u_0$ and a subsequence $\{n_k\}_{k\in\N}$
such that
\begin{equation}\label{eq:strongconv1}
\lim_{k\to\infty}|S_{\alpha_{n_k}}(t)u_{0,{n_k}}-u(t)|=0, \qquad \forall t\in(-\infty,\infty).
\end{equation}
\end{lemma}

\begin{proof}
Since $\A_w$ is the strong attractor of the Navier-Stokes equations, Theorem \ref{thm:weakstrong}
guarantees that all solutions on $\A_w$ are strongly continuous in $\H$. This means that every Leray-Hopf
solution $u:(-\infty,\infty)\to \H$ with $u(t)\in \B_0$ for each $t\in\R$ is strongly continuous.

By the invariance of the attractors $\A_{\alpha_n}$ with respect to $S_{\alpha_{n}}(t)$, we deduce (for each $n\in\N$) the existence
of a global solution $u_n:(-\infty,\infty)\to \V_{\alpha_n}$ to \eqref{eq:NSVABS}, with $u_n(t)=S_{\alpha_n}(t)u_{0,n}\in \A_{\alpha_n}$
for every $t\in\R$. The fact that $\A_{\alpha_n}\subset \B_{\alpha_n}\subset\B_0$ yields the uniform bound
\begin{equation}\label{eq:boundinH}
\sup_{n\in\N} |u_n(t)|^2\leq \frac{2(1+\lambda_1)|g|^2}{\nu^2\lambda_1^2},  \qquad \forall t\in\R.
\end{equation}
Repeating word for word the proof of Lemma  \ref{lem:conv}, we find that there exist a
global Leray-Hopf solution $u:(-\infty,\infty)\to \H$ to \eqref{eq:NSABS}
with $u(0)=u_0$ and subsequence $\{n_k\}_{k\in\N}$ such that
$$
\lim_{k\to\infty}\d_w(u_{n_k}(t), u(t))=0, \qquad \forall t\in\R.
$$
In particular,
$$
|u(t)|^2\leq \liminf_{k\to\infty}|u_{n_k}(t)|^2\leq \frac{2(1+\lambda_1)|g|^2}{\nu^2\lambda_1^2},
$$
which implies that $u(t)\in \B_0$ for every $t\in\R$, so that $u(\cdot)$ is in fact a solution on the global attractor $\A_w$
and therefore \emph{strongly} continuous by assumption. Moreover, it is not hard to see
that, possibly up to passing to a further subsequence, $u_{n_k}(t)\to u(t)$ strongly in $\H$ for almost every $t\in \R$.
Since both the sequence and the limit are strongly continuous function, this forces everywhere convergence, namely
$$
\lim_{k\to\infty}|u_{0,{n_k}}(t)-u(t)|=0, \qquad \forall t\in(-\infty,\infty),
$$
concluding the proof.
\end{proof}
The above Lemma \ref{lem:strongconv} slightly differs from Lemma \ref{lem:conv}. Here, we only know that global
solutions on the Navier-Stokes attractor are strongly continuous, and therefore the statement involves initial
conditions $u_{0,n}\in \A_{\alpha_n}$ rather than emanating from the whole absorbing set $\B_{\alpha_n}$.

As before, this implies a similar version of Corollary \ref{cor:conv}, and it is here stated only for positive times.
The proof is left to the interested reader.
\begin{corollary}\label{cor:strongconv}
Assume that $\A_w$ is the strong attractor of $S(t)$. For every $t\geq 0$, there holds
$$
\lim_{\alpha\to0}\dist_\H(S_{\alpha}(t)\A_{\alpha}, S(t)\B_0)=0.
$$
\end{corollary}

Following line by line the proof of Theorem \ref{thm:mainweak}, we infer a new necessary condition for the weak global attractor to be strong.

\begin{proposition}
Assume that the weak attractor $\A_w$ to the Navier-Stokes equations is strong. Then
\begin{equation}\label{eq:strongAttr}
\lim_{\alpha\to 0}\dist_{\H}(\A_\alpha,\A_w)=0.
\end{equation}
\end{proposition}
One may wonder if the converse is also true, namely, does the convergence $\A_\alpha\to\A_w$ as $\alpha\to 0$ implies that
$\A_w$ is the strong attractor of the Navier-Stokes equations? A partial answer is given by the following trivial observation.

\begin{proposition}
Assume that
\begin{equation}\label{eq:strongAttr2}
\lim_{\alpha\to 0}\max\left\{\dist_{\H}(\A_\alpha,\A_w),\dist_{\H}(\A_w,\A_\alpha)\right\}=0.
\end{equation}
Then the weak attractor $\A_w$ to the Navier-Stokes equations is strong.
\end{proposition}

\begin{proof}
Condition \eqref{eq:strongAttr2} immediately implies that $\A_w$ is strongly compact. Indeed, recall that
the Kuratowski measure of noncompacteness $\kur(\A_w)$ in $\H$ of $A_w$ is defined as
$$
\kur(\A_w)=\inf\big\{\delta\, :\, \A_w\ \text{has a finite cover by balls of } \H \text{ of diameter less than } \delta\big\}.
$$
It is well known that the Kuratowski measure is Lipschitz-continuous with respect to the symmetric
Hausdorff distance between sets. Therefore, since for each $\alpha\in (0,1]$ the attractors $\A_\alpha$
are compact in $\H$ (namely, $\kur(\A_\alpha)=0$), we have
$$
\kur(\A_w)=|\kur(\A_w)-\kur(\A_\alpha)|\leq c\max\left\{\dist_{\H}(\A_\alpha,\A_w),\dist_{\H}(\A_w,\A_\alpha)\right\}, \qquad \forall \alpha\in (0,1].
$$
In view of  \eqref{eq:strongAttr2}, we find $\kur(\A_w)=0$ as well. Hence, $\A_w$ is strongly compact in $\H$ and the conclusion
follows by Theorem \ref{thm:weakstrong}.
\end{proof}

\appendix

\section{Exponential attractors}\label{app:exp}
\noindent We here recall a result, first devised in \cite{EMZ}. Let $Y\Subset X$ be compactly embedded Banach
spaces, and let $S(t):X\to X$ be a semigroup. Finally, let $\B\subset X$ be a bounded closed invariant set.
\begin{theorem}\label{lem:suffexp}
Assume there exists a time $t_*>0$ such that the following hold:
\begin{enumerate}
	\item the map
	$$
	(t,x)\mapsto S(t)x: [0,t_*]\times \B\to \B
	$$
	is Lipschitz continuous (with the metric inherited from $X$);
	
	\item the map $S(t_*):\B\to \B$ admits a decomposition of the form
	$$
	S(t_*)=V+W, \qquad V:\B\to X, \qquad W:\B\to Y,
	$$
	where $V$ and $W$ satisfy the conditions
	$$
	\| V(x_1)-V(x_2)\|_X\leq \frac18\|x_1-x_2\|_X, \qquad \forall x_1,x_2\in \B,
	$$
	and
	$$
	\| W(x_1)-W(x_2)\|_Y\leq C_*\|x_1-x_2\|_X,  \qquad \forall x_1,x_2\in \B,
	$$
	for some $C_*>0$.
\end{enumerate}
Then there exist an invariant compact set $\E\subset X$ such that
$$
\dist_X(S(t)\B,\E)\leq J_0\e^{-\frac{\log 2}{t_*}t},
$$
where
$$
J_0= 2L_*\sup_{x\in \B}\|x\|_X\e^{\frac{\log 2}{t_*}}
$$
and $L_*$ is the Lipschitz constant of the map $S(t_*):\B\to \B$. Moreover, the fractal dimension of $\E$
can be estimated as
$$
\dim_X \E\leq 1+\frac{\log N_*}{\log 2},
$$
where $N_*$ is the minimum number of $\frac{1}{8 C_*}$-balls of $X$ necessary to cover
the unit ball of $Y$.
\end{theorem}

\noindent {\textbf{Acknowledgments.}}
The first author would like to express his gratitude to Professors H. Bercovici and E.S. Titi for
being a continuous source of inspiration. The work of the first author was also supported in part by
the NSF Grant DMS 1206438, and by the Research Fund of Indiana University.




\end{document}